\documentclass[12pt,a4paper]{article}
\usepackage[english]{babel}

\usepackage{pstricks}
\usepackage{pst-node}
\usepackage{amsmath,boxedminipage}
\usepackage{amssymb}
\usepackage{graphicx}
\usepackage{enumerate}
\usepackage{color}
\usepackage[text={15cm,24cm}]{geometry}

\newenvironment{proof}{{\bf Proof}:\ }%
{~\ \hfill $\Box$\vspace{0,5cm}}

{~\ \hfill$\Diamond$\vspace{0,5cm}}

\newtheorem{prop}{Property}[section]
\newtheorem{theorem}{Theorem}[section]
\newtheorem{rmk}{Remark}[section]

\newtheorem{lemma}[theorem]{Lemma}

\newtheorem{coro}[theorem]{Corollary}

\newrgbcolor{lightlightlightgray}{0.9 0.9 0.9}

\numberwithin{equation}{section}

\begin{document}

\title{On Minimum Dominating Sets in cubic and $(claw,H)-free$ graphs}
\author{
Valentin Bouquet\footnotemark[1]
\and
Fran\c cois Delbot\footnotemark[2]
\and
Christophe\ Picouleau \footnotemark[1]
\and
St\'ephane Rovedakis\footnotemark[1]
}

\date{\today}

\footnotetext[1]{ \noindent
Conservatoire National des Arts et M\'etiers, CEDRIC laboratory, Paris (France). Email: {\tt
valentin.bouquet@lecnam.net,christophe.picouleau@cnam.fr,stephane.rovedakis@cnam.fr}
}

\footnotetext[2]{ \noindent
Sorbonne Universit\'e, Laboratoire d'Informatique de Paris 6 (LIP6), Paris (France). Email: {\tt
francois.delbot@lip6.fr
}}

\graphicspath{{.}{graphics/}}

\maketitle
\begin{abstract}
Given a graph $G=(V,E)$, $S\subseteq V$ is a dominating set if  every $v\in V\setminus S$  is adjacent to an element of $S$. The Minimum Dominating Set problem asks for a dominating set with minimum cardinality. It is well known that its decision version is $NP$-complete even when $G$ is a claw-free graph. We give a complexity dichotomy for the Minimum Dominating Set problem for the class of $(claw, H)$-free graphs when $H$ has at most six vertices. In an intermediate step we show that the Minimum Dominating Set problem is $NP$-complete for cubic graphs.

\vspace{0.2cm}
 \end{abstract}

\parindent=0cm
\section{Introduction}
We will only be concerned with simple undirected graphs. The reader is referred to \cite{Bondy} and  \cite{GJ} for, respectively,  the definitions and notations on graph theory  and on computational complexity.

Given a graph $G=(V,E)$, a set $S\subseteq V$ is called a {\it dominating set} if every vertex $v\in V$ is either an element of $S$ or is adjacent to an element of $S$. When $S$ is a dominating set of $G$ we say that $S$ {\it dominates} $G$. The minimum cardinality of a dominating set in $G$ is denoted by $\gamma(G)$. A dominating set $S\in V$ with $\vert S\vert=\gamma(G)$ is called a {\it Minimum Dominating Set}, a mds for short. Following \cite{DomBook} a mds is also called a $\gamma$-set.

Our aim is to determine the computational complexity of the task consisting of computing a $\gamma$-set or the domination number, for some subclasses of graphs defined by a finite set of forbidden subgraphs.

The decision problem associated with the Minimum Dominating Set is defined as:
\begin{center}
\begin{boxedminipage}{.99\textwidth}
\textsc{\sc Minimum Dominating Set} {\small(MDS problem)} \\[2pt]
\begin{tabular}{ r p{0.8\textwidth}}
\textit{~~~~Instance:} &a graph $G=(V,E)$ and an integer  $d,\vert V\vert \geq d\geq 0$.\\
\textit{Question:} &is $\gamma(G)\leq d$ ?
\end{tabular}
\end{boxedminipage}
\end{center}

In this paper we focus on the $(claw, H)$-free graphs complexity. The decision problems we study are defined this way.

\begin{center}
\begin{boxedminipage}{.99\textwidth}
\textsc{\sc $(claw,H)$-free Minimum Dominating Set}\\[2pt]
\begin{tabular}{ r p{0.8\textwidth}}
\textit{~~~~Instance:} &a $(claw,H)$-free graph $G=(V,E)$ and an integer $d,\vert V\vert\ge d\geq 0$.\\
\textit{Question:} &is $\gamma(G)\leq d$ ?
\end{tabular}
\end{boxedminipage}
\end{center}

The paper is organized as follows. The two next sections give the notations, the results of the literature and some basic properties that will be used in the sequel of the paper. Then we prove that the MDS problem is $NP$-complete for the class of cubic graphs, a result that is strangely missing in the literature. This result and its proof will be useful for several demonstrations later on. The sections \ref{C4toCk} to \ref{twotriangles} are concerned with the MDS problem in the class of claw-free graphs when at least one other graph is excluded. Among our different results,  in the section \ref{clawH} we give a complexity dichotomy for the class of $(claw,H)$-free graphs for all the graphs $H$ with no more than six vertices. We give a partial result for the $(claw,H)$-free MDS problem when $H$ has at least seven vertices  in the section \ref{twotriangles}. We summarize our main results and the problems left open in the conclusion.

\section{Definitions and notations}
An element $ab\in E$ is called an {\it edge}, if $ab\not\in E$ then $ab$ is called a {\it non-edge}. For a vertex $v\in V$ let us denote by $N(v)$ its neighborhood, $N[v]=N(v)\cup\{v\}$ its closed neighborhood. The set of vertices at distance exactly $k$ of a vertex $v$ is denoted by $N_k(v)$. Hence $N(v)=N_1(v)$ and $N[v]=N_0(v)\cup N_1(v)$. A vertex $v$ is {\it isolated} if $N(v)=\emptyset$. A vertex $v$ is {\it universal} if $N[v]=V$. Two distinct vertices $u,v$ are {\it twins} if $N(v)=N(u)$, are {\it false twins} if $N[v]=N[u]$. The graph $\overline{G}$ is the complementary graph of $G$, that is $V(G)=V(\overline{G})$ and $E(\overline{G})=\{uv\; |\; uv\not\in E(G)\}$.

For a subset $S\subseteq V$, we let $G[S]$ denote the subgraph of $G$ {\it induced} by $S$, which has vertex set~$S$ and edge set $\{uv\in E\; |\; u,v\in S\}$. Moreover, for a vertex $v\in V$, we write $G-v=G[V\setminus \{v\}]$ and for a subset $V'\subseteq V$ we write $G-V'=G[V\setminus V']$.
For a set $\{H_1,\ldots,H_p\}$ of graphs, $G$ is {\it $(H_1,\ldots,H_p)$-free} if $G$ has no induced subgraph isomorphic to a graph in $\{H_1,\ldots,H_p\}$; if $p=1$ we may write $H_1$-free instead of $(H_1)$-free. For two vertex disjoint induced subgraphs $G[A],G[B]$, $G[A]$ is {\it complete} to $G[B]$ if $ab$ is an edge for any $a\in A$ and $b\in B$, $G[A]$ is {\it anti-complete} to $G[B]$ if $ab$ is an non-edge for any $a\in A$ and $b\in B$. We denote by $G+H$ the disjoint union of the graphs $G$ and $H$.

A set $S\subseteq V$ is called a {\it stable set} or an {\it independent set} if any pairwise distinct vertices $u,v\in S$ are non adjacent. The maximum cardinality of an independent set in $G$ is denoted by $\alpha(G)$. A set $S\subseteq V$ is called a {\it clique} if any pairwise distinct vertices $u,v\in S$ are adjacent. When $G[V]$ is a clique then $G$ is a {\it complete graph}. We denote by $K_p,p\ge 1,$ the clique or the complete graph on $p$ vertices and $k.K_p$ is the disjoint union of $k$ cliques ($k \geq 0$).

For $n\geq 1$, the graph $P_n=u_1-u_2-\cdots-u_n$ denotes the {\it cordless path} on $n$ vertices, that is, $V({P_n})=\{u_1,\ldots,u_n\}$ and $E({P_n})=\{u_iu_{i+1}\; |\; 1\leq i\leq n-1\}$.
For $n\geq 3$, the graph $C_n$ denotes the {\it cordless cycle} on $n$ vertices, that is, $V({C_n})=\{u_1,\ldots,u_n\}$ and $E({C_n})=\{u_iu_{i+1}\; |\; 1\leq i\leq n-1\}\cup \{u_nu_1\}$. For $n\ge 4$, $C_n$ is called a {\it hole}. The graph $C_3=K_3$ is also called the {\it triangle}.
The {\it claw} $K_{1,3}$ is the 4-vertex star, that is, the graph with vertices $u$, $v_1$, $v_2$, $v_3$ and edges $uv_1$, $uv_2$, $uv_3$. The {\it diamond} is the 4-vertex complete graph $K_4$ minus an edge. The {\it net} is the graph with six vertices $u_1,u_2,u_3$, $v_1$, $v_2$, $v_3$ and edges $u_1u_2,u_2u_3,u_1u_3,u_1v_1,u_2v_2,u_3v_3$. The {\it bull} is the graph with five vertices $u_1,u_2,u_3$, $v_1$, $v_2$ and edges $u_1u_2,u_2u_3,u_1u_3,u_1v_1,u_2v_2$. The {\it paw} is the graph with four vertices $u_1,u_2,u_3$, $v_1$ and edges $u_1u_2,u_2u_3,u_1u_3,u_1v_1$. The {\it butterfly} is the graph with five vertices $u_1,u_2,v,v_1,v_2$ and edges $u_1u_2, u_1v, u_2v, v_1v_2, v_1v, v_2v$. The {\it house} is the graph with five vertices $u_1,u_2,u_3,u_4,v$ and edges $u_1u_2, u_2u_3, u_3u_4, u_4u_1, u_1v, u_2v$. The {\it gem} is the graph with five vertices $u_1,u_2,u_3,u_4,v$ and edges $u_1u_2, u_2u_3, u_3u_4$, $u_1v, u_2v, u_3v, u_4v$. The Figure \ref{defgraphs} below exposed all these graphs.

\begin{figure}[htbp]
\begin{center}
\includegraphics[width=10cm, height=5cm, keepaspectratio=true]{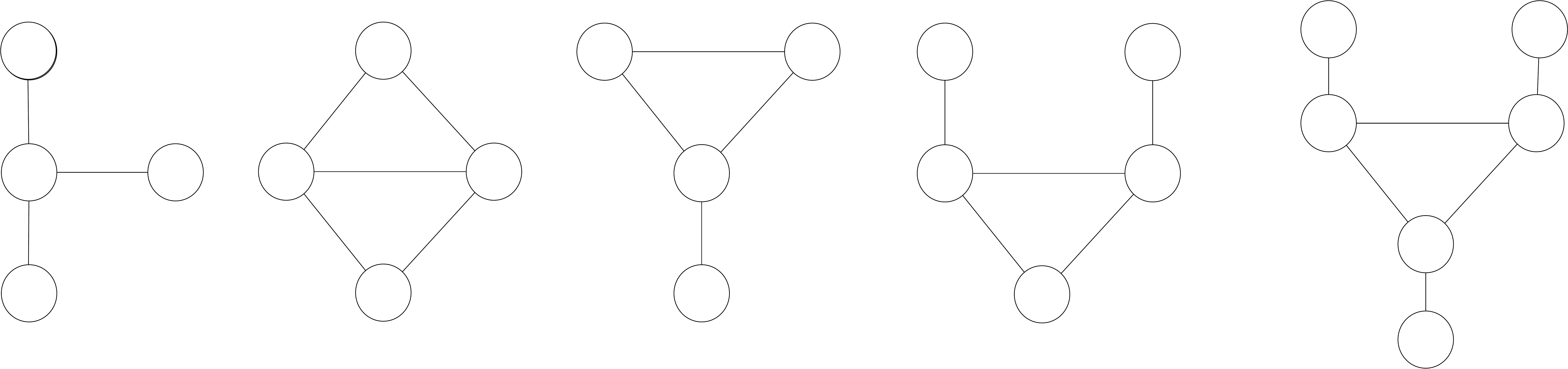}
\end{center}
\caption{The claw, the diamond, the paw, the bull, the net.}
\label{defgraphs}
\end{figure}

\begin{figure}[htbp]
\begin{center}
\includegraphics[width=6cm, height=5cm, keepaspectratio=true]{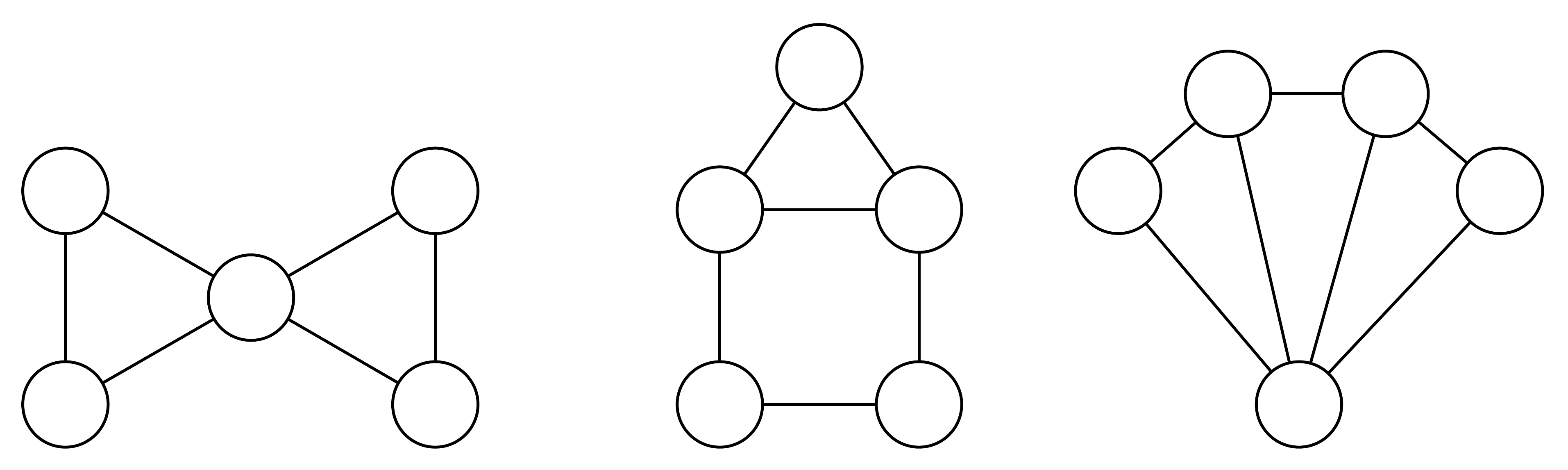}
\end{center}
\caption{The butterfly, the house, the gem.}
\label{defgraphs}
\end{figure}

The {\it $(k_1,k_2,k_3)$-triangle} consists of a triangle $T=\{v_1,v_2,v_3\}$ and three vertex disjoint paths $P_{k_i}$ connected to $v_i,1\le i\le 3$. Hence the net is the $(1,1,1)$-triangle, the bull is the $(1,1,0)$-triangle, the paw is the $(1,0,0)$-triangle. The {\it $k$-double-triangle} consists of two vertex disjoint triangles $T_1=\{v_1,v_2,v_3\},T_2=\{v_4,v_5,v_6\}$ and $P_k$ between $v_1$ and $v_4$. For simplicity the $0$-double-triangle is called the double-triangle (see Figure \ref{knets}).

\begin{figure}[htbp]
\begin{center}
\includegraphics[width=12cm, height=5cm, keepaspectratio=true]{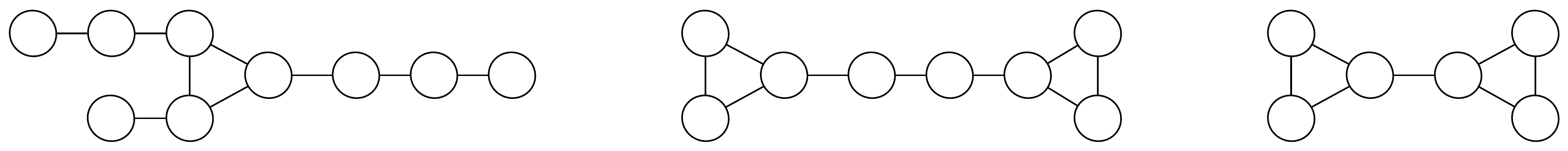}
\end{center}
\caption{The $(3,2,1)$-triangle, the $2$-double-triangle, the double-triangle.}
\label{knets}
\end{figure}

We denote $V^+\subseteq V$ the subset of vertices $v$ of $G$ such that $\gamma(G-v)>\gamma(G)$. A mds which is also an independent set is an {\it independent dominating set} and the minimum cardinality of an independent dominating set is denoted by $i(G)$. Clearly $\gamma(G)\le i(G)\le\alpha(G)$. Note also that a minimum independent dominating set is a {\it minimum maximal independent set}. \\

\section{Preliminary results}

We know from \cite{IndDom} that if the graph $G$ is $claw$-free then $\gamma(G)=i(G)$. From \cite{Yannakakis} the MDS problem
is $NP$-complete for the clas of $claw$-free graphs.

The minimum edge dominating set problem consists of finding $F\subseteq E$ a minimum set of edges such that for each edge $e\not\in F$, $e$ is incident to an edge $f, f\in F$. Taking $L(G)$, the line graph of $G$, a minimum edge dominating set in $G$ is a minimum dominating set in $L(G)$.
In \cite{Yannakakis} Yannakakis et al. showed that the minimum edge dominating set problem is $NP$-complete for bipartite subcubic graphs. Moreover, one can check that the graphs built in the transformation (from a variant of the $3$-SAT problem) are $C_4$-free. Also, line graphs of bipartite graphs are perfect, so they have no odd holes. Thus, for any of these graphs, the corresponding line graphs are $(claw, diamond, K_4,C_4)$-free and perfect (i.e. $(odd\ hole, odd\ antihole)$-free). It follows that the minimum dominating set problem is $NP$-complete for $(claw, diamond, K_4,C _4)$-free perfect graphs.

The minimum dominating set problem is polynomial for $\overline {paw}$-free graphs. The class of $(claw, \overline {claw})$-free graphs has bounded clique-width \cite{clawcoclaw}, so from \cite{cliquewidth} a $\gamma$-set can be computed in linear time. It is also polynomial for $(claw, net)$-free graphs \cite{clawnet}. A minimum dominating set can be computed in polynomial time for $2K_2$-free graphs so for $(claw, 2K_2)$-free graphs \cite{IndDomPoly}. In \cite{clawpk}, it is shown that computing a minimum dominating set is polynomial for the class of $(claw, P_8)$-free graphs. Since $K_2+2K_1\subseteq_i P_8$ and $4K_1\subseteq_i P_8$ computing a mds is polynomial for the classes of $(claw, K_2+2K_1)$-free and $(claw, 4K_1)$-free graphs.\\

Note that $(claw, H)$-free graphs is equivalent to $claw$-free graphs when $H$ contains a $claw$. Hence, the remaining $H$'s we consider are $claw$-free.\\

We give some preliminary easy properties that will be useful for many proofs given later.
\begin{prop}\label{HinH'NPC}
Let $H$, $H'$ two fixed graphs with $H' \subseteq_i H$. If \textsc{\sc $H'$-free Minimum Dominating Set}
 is $NP-complete$ then \textsc{\sc $H$-free Minimum Dominating Set} is $NP-complete$.
\end{prop}

\begin{prop}\label{HinH'POLY}
Let $H$, $H'$ two fixed graphs with $H \subseteq_i H'$. If \textsc{\sc $H'$-free Minimum Dominating Set}
 is $polynomial$ then \textsc{\sc $H$-free Minimum Dominating Set} is $polynomial$.
\end{prop}

\begin{prop}\label{V=N[T]}
Let $k>0$ be a fixed positive integer and $G=(V,E)$ a graph. If there exists $T \subseteq_i G$ of size $\vert T\vert\leq k$ such that $V=N[T]$ then computing a minimum dominating set for $G$ is polynomial.
\end{prop}
\begin{proof}
Since $\vert T \vert\leq k$ and $V=N[T]$, we have that $\gamma(G) \leq k$. So a minimum dominating set can be computed in $O(n^k)$ which is polynomial.
\end{proof}

\begin{prop}\label{kK1}
Let $k>0$ be any fixed positive integer. If a graph $G$ is $(claw,k.K_1)$-free then computing a minimum dominating set for $G$ is polynomial.
\end{prop}
\begin{proof}
Since $G$ is $k.K_1$-free we have $\alpha(G)<k$. Since $G$ is claw-free $\gamma(G)=i(G)\le\alpha(G)$. Now $k$ is fixed, so enumerating all the independent sets of size less than $k$ can be done in polynomial time. It follows that for any fixed $k\ge 1$, \textsc{\sc $(claw,k.K_1)$-free Minimum Dominating Set} is polynomial.
\end{proof}

\begin{prop}\label{leaf}
Let $G=(V,E)$ be a connected $claw$-free graph with $uv\in E$ such that $u$ is a leaf. The \textsc{Minimum Dominating Set} problem is polynomial for $G'=G-N[v]$ if and only if it is polynomial for $G$.
\end{prop}
\begin{proof}
Since $G$ is $claw$-free $K=N(v)-u$ is a clique. Let $s\in K$, we show that $s\not \in V^+$. For contradiction we assume that $s \in V^+$. As shown in \cite{DomAlter}, $s$ is in every $\gamma$-set of $G$. Let $\Gamma$ be a $\gamma$-set with $v\in \Gamma$. Let $W=N(s)\setminus N(v)$. If $W=\emptyset$ then $\Gamma-s$ is a dominating set, a contradiction. It follows $W\not=\emptyset$ and since $G$ is $claw$-free, $W$ is a clique. Let $w\in W$. $\Gamma'=(\Gamma-\{s\})\cup\{w\}$ is another $\gamma$-set, a contradiction. Hence $\gamma(G-K)\le\gamma(G)$. Since $G-K$ consists of $G'$ and the component $uv$, we have that $\gamma(G-K)=\gamma(G')+1$. Then from $\Gamma'$ a $\gamma$-set of $G'$ we obtain $\Gamma'\cup\{v\}$ a $\gamma$-set of $G$ in polynomial time. Reciprocally, let $\Gamma$ be a $\gamma$-set of $G$. Since $u$ is a leaf we assume that $v\in \Gamma$. Then $\Gamma-\{v\}$ is a $\gamma$-set for $G'$. Trivially it can be done in polynomial time from $\Gamma$.
\end{proof}

\section{Cubic and $k$-regular graphs for $k\ge 3$ when $k$ is odd}
In \cite{Demange} Demange et al. it is proved that, for any fixed $k\ge 3$, the Minimum Edge Dominating Set problem is $NP$-complete for bipartite $k$-regular graphs. Transferring this result into the line graph we have that the Minimum Dominating Set problem is $NP$-complete for $2(k-1)$-regular graphs, $k\ge 3$, that is regular graphs with even degree at least four.

\begin{theorem}\label{cubic}
The Minimum Dominating Set problem is $NP$-complete for cubic graphs.
\end{theorem}
\begin{proof}
We give a polynomial reduction from the Minimum Dominating Set problem which is $NP$-complete for $4$-regular graphs.

Let $G$ be a $4$-regular graph and $d$ be an instance of Minimum Dominating Set. We build $G'$ and $d'$ another instance of Minimum Dominating Set where $G'$ is cubic. Let $n$ be the number of vertices of $G$, we take $d'=11n+d$. Each $4$-vertex $v$ of $G$ is replaced with the gadget shown in Figure \ref{gadcub}. We denote by $H_v$ the gadget associated with $v$.

The gadget $H_v$ has four {\it corners}, that is the vertices $a,b,c,d$ in Figure \ref{gadcub}. Each of them corresponds to an edge incident to $v$ in $G$. To the edge $a=uv$ in $G$ corresponds an edge connecting the two corners $a$ of $H_v$ and $H_u$. Hence $G'$ is cubic.

For each $H_v$, the subgraph of $G'$ induced by the $44$ vertices of $H_v$, satisfies the following properties. Note that by symmetry the four corners play the same role. Clearly $\gamma(H_v)=12$ and there exists a minimum dominating set that contains the four corners, see Figure \ref{gadgamma}. For each corner, says $a$, $H_v-a$ has a unique minimum dominating set of size $\gamma(H_v-a)=11$, see Figure \ref{gadpartdom}. None of the three other corners $b,c,d$ are in this minimum dominating set. With ten vertices it is not possible to dominate $H_v-\{a,b,c,d\}$.

\begin{figure}[htbp]
\begin{center}
\includegraphics[width=14cm, height=7cm, keepaspectratio=true]{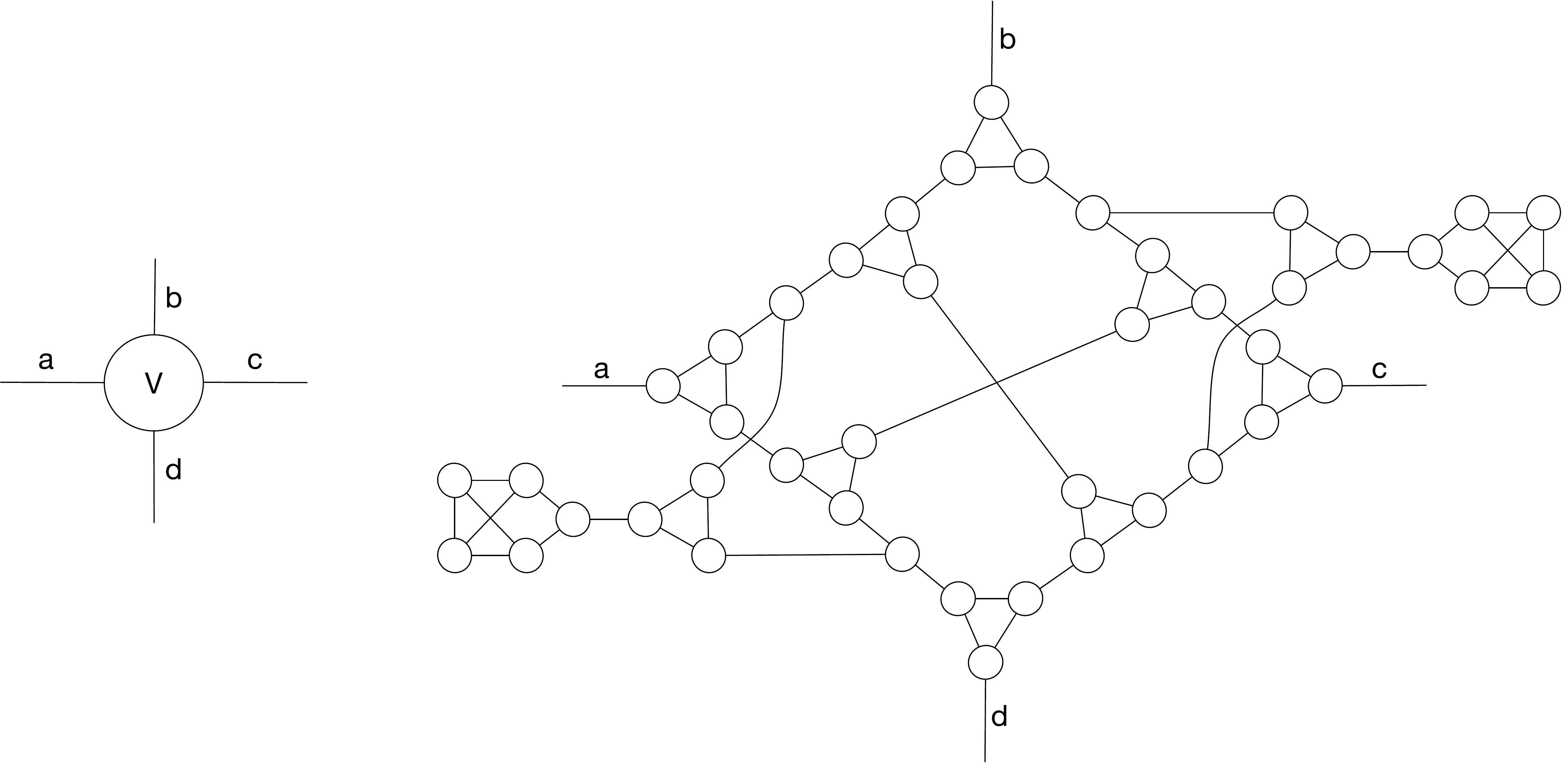}
\end{center}
\caption{A $4$-vertex $v$ and its associated gadget $H_v$.}
\label{gadcub}
\end{figure}

Let $\Gamma$ be a dominating set of $G$ such that $\vert \Gamma\vert\le d$. We give $\Gamma'$ a dominating set of $G'$ as follows: if $v\in \Gamma$ then the twelve vertices of the minimum dominating set of $H_v$ (the $12$ black vertices of Figure \ref{gadgamma}), are put in $\Gamma'$. If $v\not\in \Gamma$ then $v$ is dominated by a vertex $w, w\in \Gamma$. Let $a=vw$. The eleven vertices of the minimum dominating set of $H_v-a$ (the $11$ black vertices of Figure \ref{gadpartdom}), are put in $\Gamma'$.
Hence the corner $a$ of $H_v$ is dominated by the corner $a$ of $H_w$. So $\Gamma'$ is a dominating set of size $\vert \Gamma'\vert=11n+d=d'$.

\begin{figure}[htbp]
\begin{center}
\includegraphics[width=14cm, height=7cm, keepaspectratio=true]{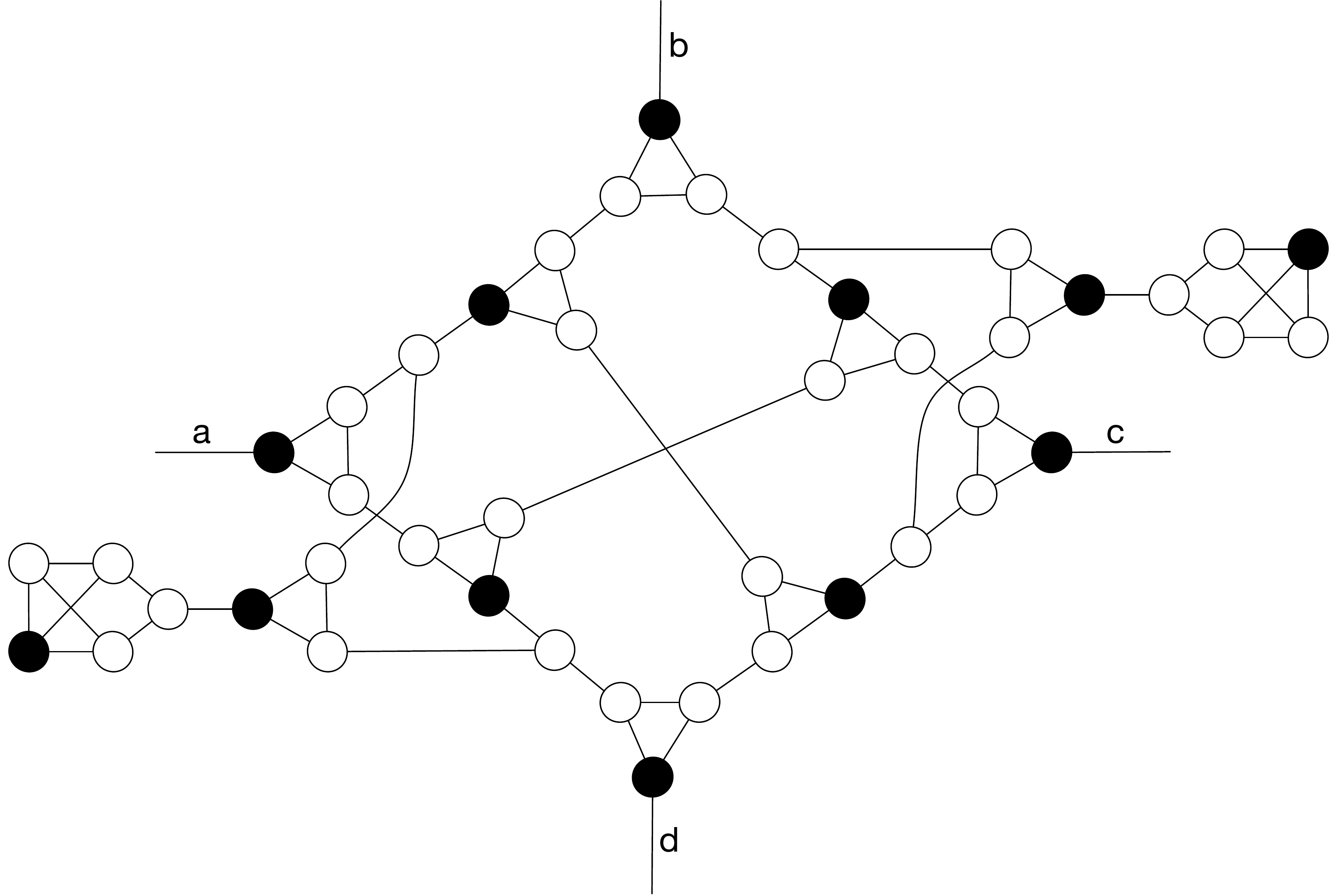}
\end{center}
\caption{A minimum dominating set for the gadget.}
\label{gadgamma}
\end{figure}

\begin{figure}[htbp]
\begin{center}
\includegraphics[width=14cm, height=7cm, keepaspectratio=true]{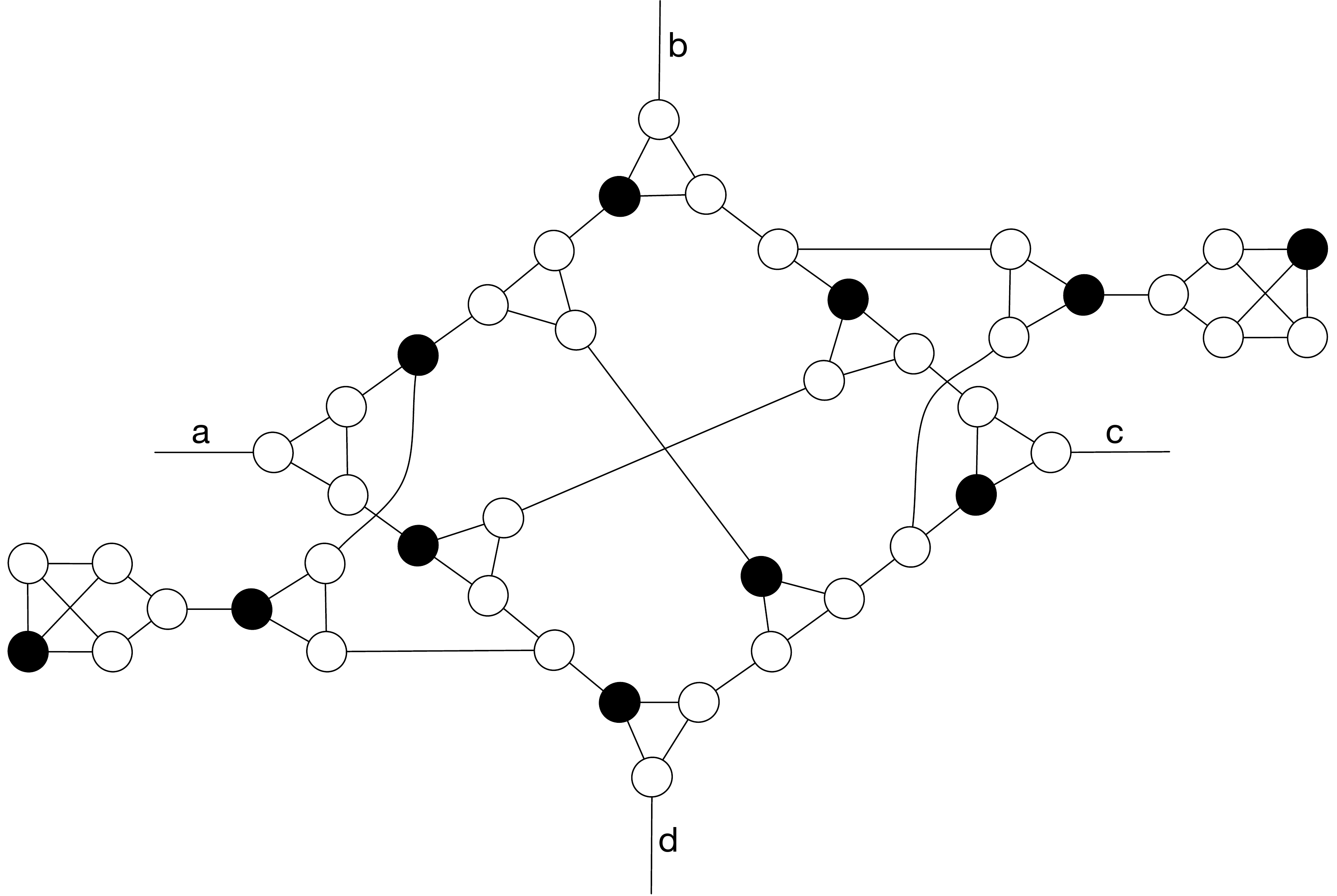}
\end{center}
\caption{A partial dominating set of size $11$ for the gadget.}
\label{gadpartdom}
\end{figure}
Let $\Gamma'$ be a dominating set of $G'$ such that $\vert \Gamma\vert\le 11n+d=d'$. We show that $\gamma(G)\le d$. From the last property, we gave at least eleven vertices to dominate each $H_v$ (note that it is necessary to dominate $H_v-\{a,b,c,d\}$). Since $d'=11n+d'$ at most $d$ gadgets $H_v$ must contain at least $12$ vertices. Since twelve vertices are enough for one $H_v$, we can assume there are $d$ gadgets with $\vert H_v\cap \Gamma'\vert=12$. Taking the $d$ corresponding vertices $v$ in $\Gamma$, $\Gamma$ is a dominating set of $G$. Hence we have that $\gamma(G)\le\vert\Gamma\vert\le d$.
\end{proof}

The graphs built in the proof above are $(butterfly,house,gem)$-free, so we have the following.
\begin{coro}\label{cubcoro}
 The Minimum Dominating Set problem is $NP$-complete for $(K_4,$ $butterfly, house,gem)$-free cubic graphs.
\end{coro}
\begin{theorem}\label{oddregular}
For any odd $k\ge 3$, the Minimum Dominating Set problem is $NP$-complete for $k$-regular graphs.
\end{theorem}
\begin{proof}
\begin{figure}[htbp]
\begin{center}
\includegraphics[width=16cm, height=5cm, keepaspectratio=true]{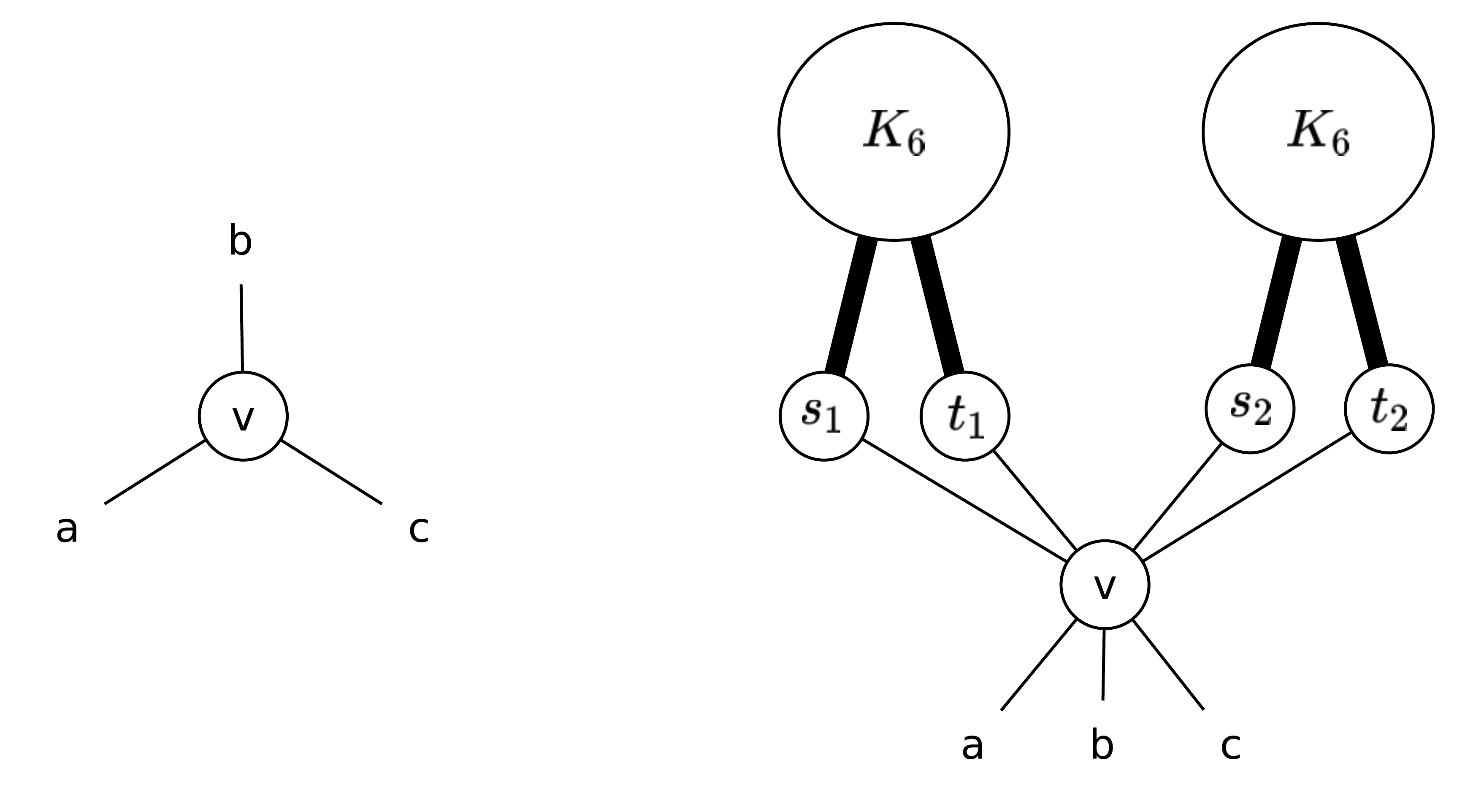}
\end{center}
\caption{A $3$-vertex $v$ and its replacement in the case of $7$-regular graphs. The bold edge $s_1K_6$ means that $s_1$ is connected to the $6$ vertices of $K_6$.}
\label{5reg}
\end{figure}
The case $k=3$ is proved by Theorem \ref{cubic}. Now let $k\ge 5$ be an odd integer.
Given a cubic graph $G=(V,E)$, we construct a graph $G'=(V',E')$ as follows. Each vertex $v, v\in V',$ is connected to $(k-3)\over 2$ components $K_{k+1}-e$ (the complete graph with $k+1$ vertices minus an edge). For each component $K_{k+1}-e$, let $e=st$. The vertex $v$ is connected to $s$ and $t$. Figure \ref{5reg} shows the transformation for $k=7$. For each component, one vertex $w, w\ne s,t$ is necessary in a minimum dominating set. Since $v$ is not dominated by $w$, the sequel of the proof is easy.
\end{proof}

\section{$(claw, C_4,\ldots, C_k)$-free graphs}\label{C4toCk}
With similar reduction as the one used for proving Theorem \ref{cubic}, we show the following theorem.
\begin{theorem}\label{c4ck}
For any fixed integer $k,k\ge 4,$ the Minimum Dominating Set problem is $NP$-complete for $(claw, C_4,\ldots, C_k)$-free subcubic graphs.
\end{theorem}
\begin{proof}
The arguments are similar of those given in the proof of Theorem \ref{cubic}.

We give a polynomial reduction from the Minimum Dominating Set problem which is $NP$-complete for $4$-regular graphs. To each $4$-vertex $v$ is associated the gadget depicted in Figure \ref{bigck}. In this gadget each dashed box corresponds to an induced path of $3p$ vertices, $p \geq 0$.
\begin{figure}[htbp]
\begin{center}
\includegraphics[width=14cm, height=7cm, keepaspectratio=true]{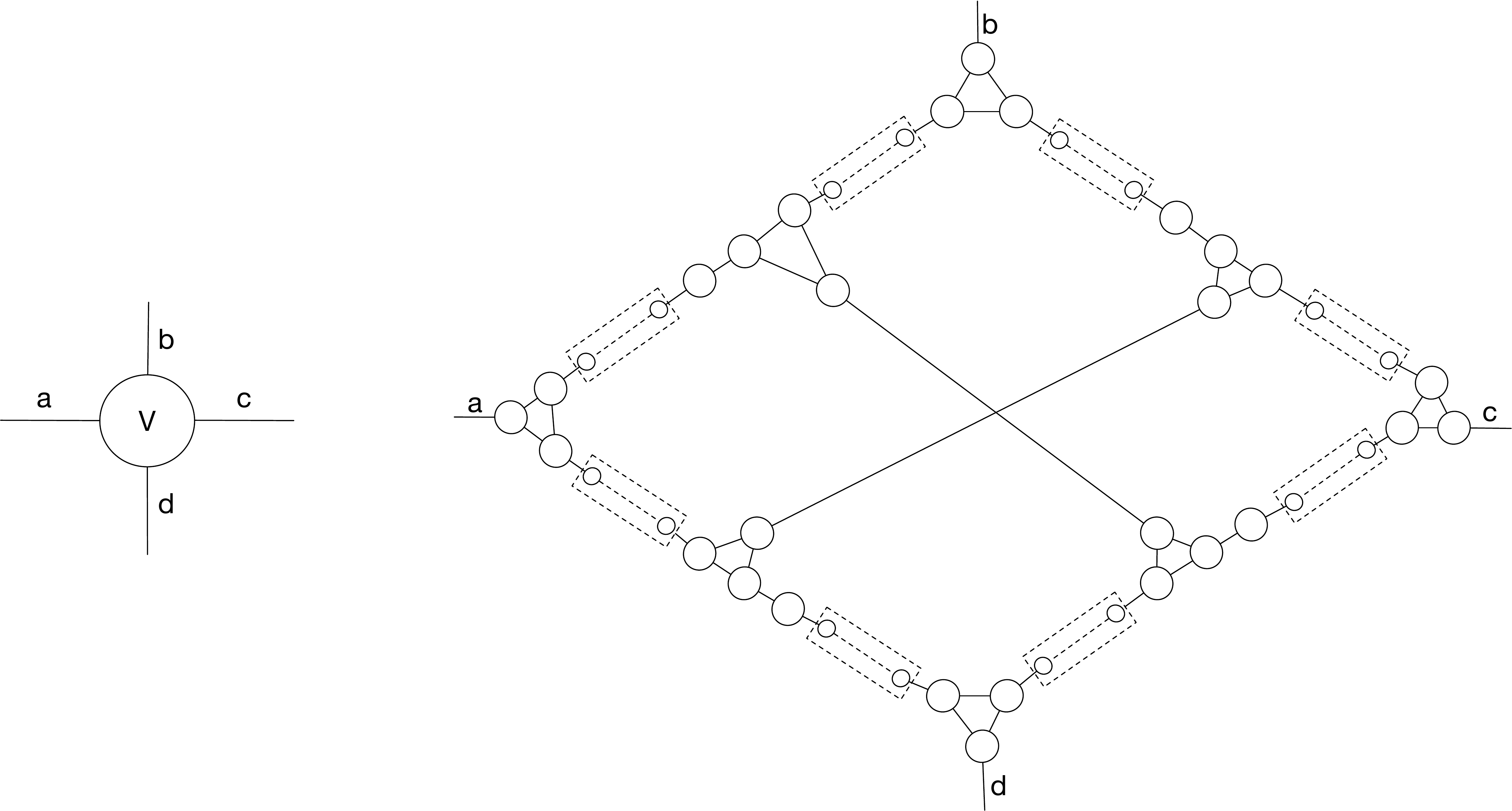}
\end{center}
\caption{A $4$-vertex $v$ and its associated gadget $H_v$.}
\label{bigck}
\end{figure}

We have the following: $\gamma(H_v)=8+4p$ and there exists a minimum dominating set that contains the four corners. For each corner, says $a$, $H_v-a $ has a unique minimum dominating set of size $\gamma(H_v-a)=7+4p$. None of the three other corners $b,c,d$ are in this minimum dominating set. With $6+4p$ vertices it is not possible to dominate $H_v-\{a,b,c,d\}$.

The graph we obtain has no claw. Its maximum degree is three. The smallest cycle that is not a triangle has length more than $12p$. As $p$ can be chosen arbitrary large taking $k=12p$ we obtain the result.
\end{proof}

\section{$(claw,k-double-triangle)$-free graphs}
Using the Theorem \ref{cubic}, we show the $NP$-completeness of the MDS problem for the class of $(claw,\ k-double-triangle)$-free subcubic graphs. In a first step we prove the following lemma.
\begin{lemma}\label{butterfly}
The $(claw, butterfly,diamond,C_4,C_5,K_4)$-free Minimum Dominating Set problem is $NP$-complete for cubic perfect graphs.
\end{lemma}
\begin{proof}
We give a polynomial reduction from the Minimum Dominating Set problem for cubic graphs which is proved $NP$-complete in Theorem \ref{cubic}.

Let $G$ be a cubic graph and $d$ be an instance of Minimum Dominating Set. We build $G'$ and $d'$ another instance of Minimum Dominating Set where $G'$ is $(claw,butterfly,\\ diamond,C_4,C_5,K_4)$-free and perfect. Let $n$ be the number of vertices in $G$, we take $d'=2n+d$. Each $3$-vertex $v$ of $G$ is replaced with the gadget shown in Figure \ref{gadbutter}. We denote by $H_v$ the gadget associated with $v$.

The gadget $H_v$ has three {\it corners} - the vertices $a,b,c$ in Figure \ref{gadbutter} - each of them corresponding to an edge incident to $v$ in $G$. To the edge $a=uv$ in $G$ corresponds an edge connecting the two corners $a$ of $H_v$ and $H_u$. Hence $G'$ is $(claw$, $butterfly$, $diamond$, $C_4,C_5,K_4)$-free. Moreover $G'$ has no odd hole. Since $C_5$ is isomorphic to $\overline{C_5}$ we have $C_5\not\subseteq_i \overline{G'}$. One can remark that for any $k, k\ge 6,$ we have $C_4\subseteq_i {\overline{C_k}}$. Thus $G'$ being $C_4$-free, $G'$ has no odd anti-hole. Then from the perfect graph theorem \cite{perfect} $G'$ is perfect.

For each $H_v$, the subgraph of $G'$ induced by the nine vertices of $H_v$, satisfies the following properties. By symmetry the three corners play the same role. Clearly $\gamma(H_v)=3$ and there exists a minimum dominating set that contains the three corners, see Figure \ref{buttergamma}. For each corner, says $a$, $H_v-a $ has a unique minimum dominating set of size $\gamma(H_v-a)=2$, see Figure \ref{buttergamma}. None of the two other corners $b,c$ are in this minimum dominating set. With one vertex it is not possible to dominate $H_v-\{a,b,c\}$.

\begin{figure}[htbp]
\begin{center}
\includegraphics[width=10cm, height=7cm, keepaspectratio=true]{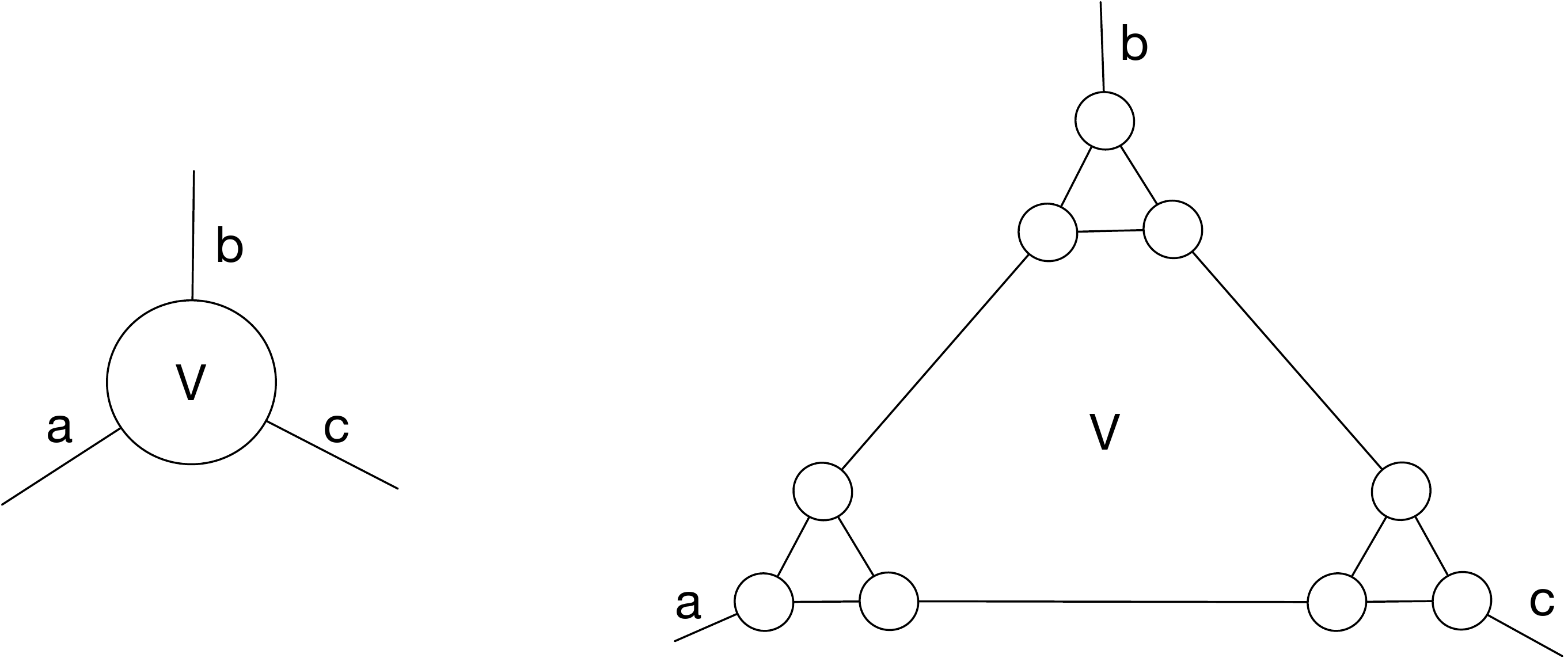}
\end{center}
\caption{A $3$-vertex $v$ and its associated gadget $H_v$.}
\label{gadbutter}
\end{figure}

\begin{figure}[htbp]
\begin{center}
\includegraphics[width=14cm, height=7cm, keepaspectratio=true]{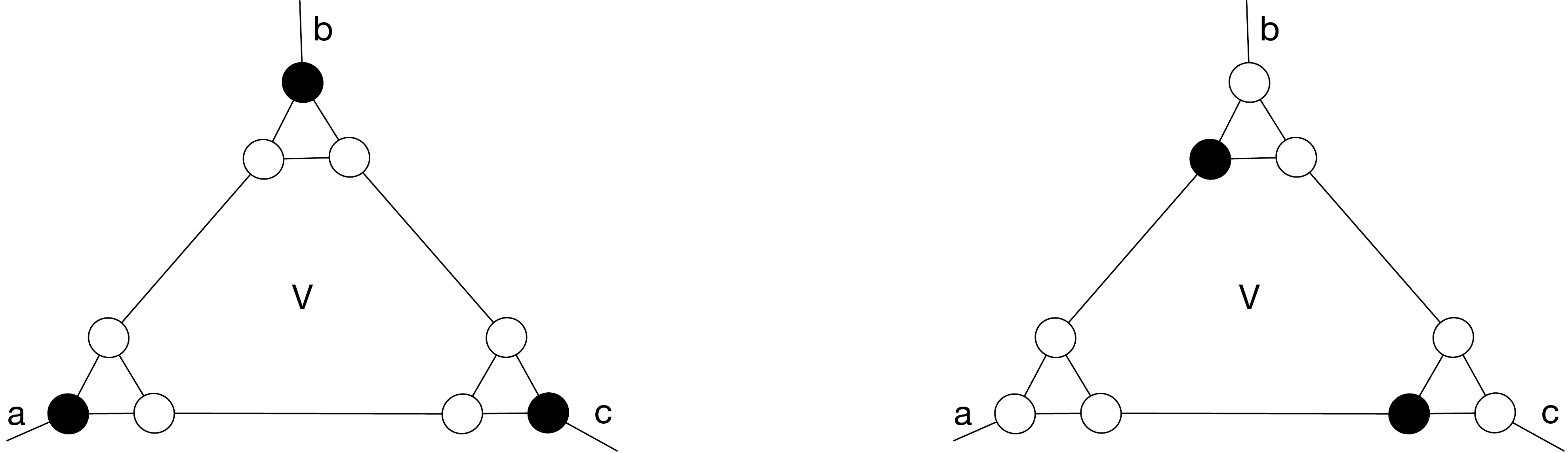}
\end{center}
\caption{A minimum dominating set and partial dominating set of size $2$ for the gadget}
\label{buttergamma}
\end{figure}
Let $\Gamma$ be a dominating set of $G$ such that $\vert \Gamma\vert\le d$. We give $\Gamma'$ a dominating set of $G'$ as follows: If $v\in \Gamma$ then the three vertices of the minimum dominating set of $H_v$ (the three black vertices of Figure \ref{buttergamma}), are put in $\Gamma'$. If $v\not\in \Gamma$ then $v$ is dominated by a vertex $w, w\in \Gamma$. Let $a=vw$. The two vertices of the minimum dominating set of $H_v-a$ (the two black vertices of Figure \ref{buttergamma}), are put in $\Gamma'$.
Hence the corner $a$ of $H_v$ is dominated by the corner $a$ of $H_w$. So $\Gamma'$ is a dominating set of size $\vert \Gamma'\vert=2n+d=d'$.\\

Let $\Gamma'$ be a dominating set of $G'$ such that $\vert \Gamma\vert\le 2n+d=d'$. We show that $\gamma(G)\le d$. From the last property we gave for $H_v$ at least two vertices are necessary to dominate $H_v-\{a,b,c\}$. Since $d'=2n+d'$ at most $d$ gadgets $H_v$ must contain at least three vertices. Since three vertices are enough for one $H_v$, we can assume there are $d$ gadgets with $\vert H_v\cap \Gamma'\vert=3$. Taking the $d$ corresponding vertices $v$ in $\Gamma$, $\Gamma$ is a dominating set of $G$. Hence we have that $\gamma(G)\le\vert\Gamma\vert\le d$.
\end{proof}

So we can prove the following theorem.
\begin{theorem}\label{kdoubletriangle}
For any fixed integer $k,k\ge 0,$ the Minimum Dominating Set problem is $NP$-complete for $(claw,\ k-double-triangle)$-free subcubic graphs.
\end{theorem}
\begin{proof}
The arguments are similar of those given in the proof of Lemma \ref{butterfly}.

We give a polynomial reduction from the Minimum Dominating Set problem which is $NP$-complete for cubic graphs. Let $G=(V,E)$ be a cubic graph. To each $3$-vertex $v,v\in V,$ is associated the gadget depicted by Figure \ref{kdoubtriangle}. In this gadget each dashed box corresponds to an induced path of $3p$ vertices (we depicted the case $p=1$). To each edge $uv\in E$ is associated a $P_{3p}$ (the same dashed box in the picture). We take $d'=d+(2+3p)\vert V\vert+p\vert E\vert$.

\begin{figure}[htbp]
\begin{center}
\includegraphics[width=10cm, height=7cm, keepaspectratio=true]{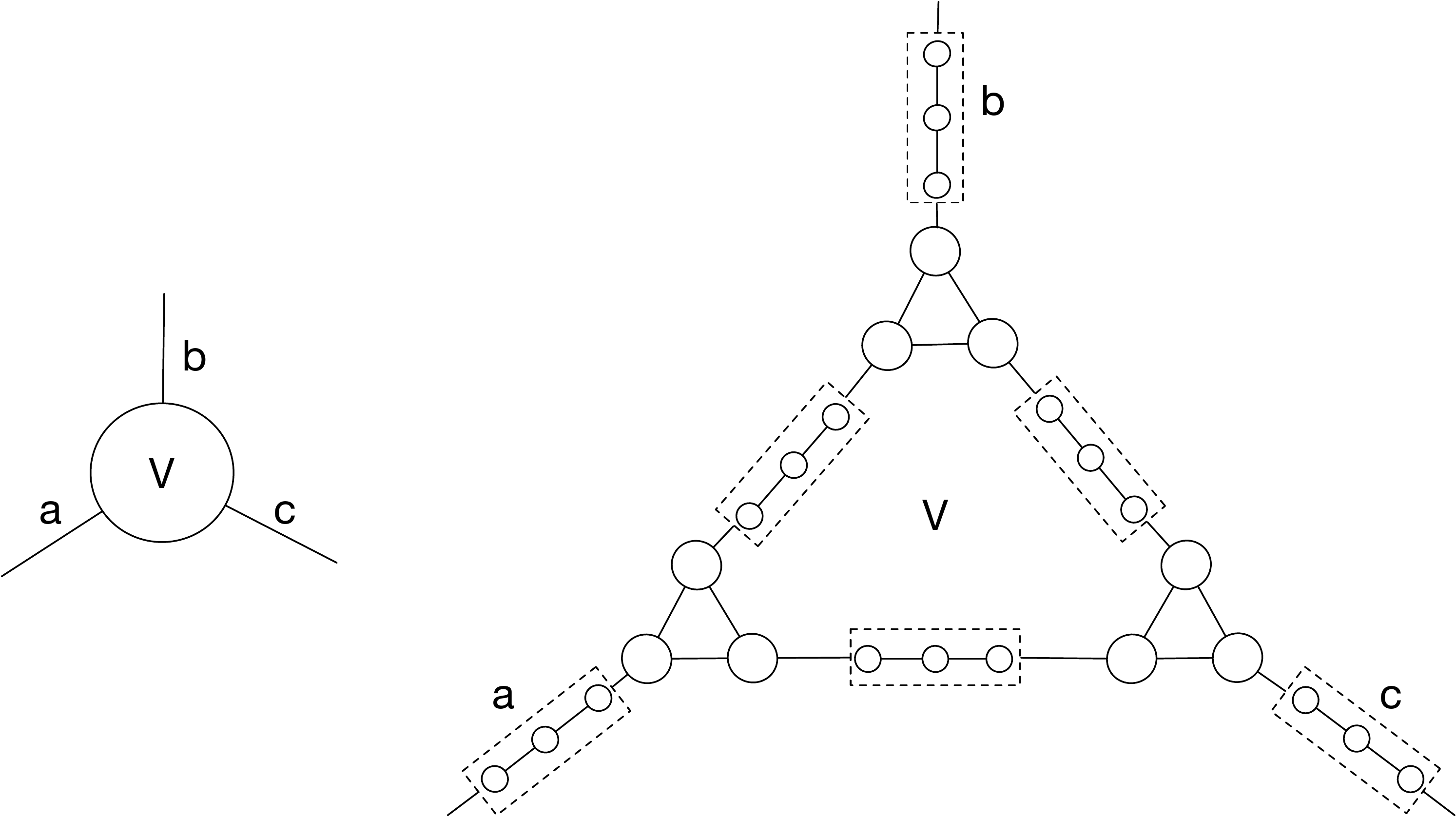}
\end{center}
\caption{A $3$-vertex $v$ and its associated gadget $H_v$.}
\label{kdoubtriangle}
\end{figure}

\begin{figure}[htbp]
\begin{center}
\includegraphics[width=14cm, height=7cm, keepaspectratio=true]{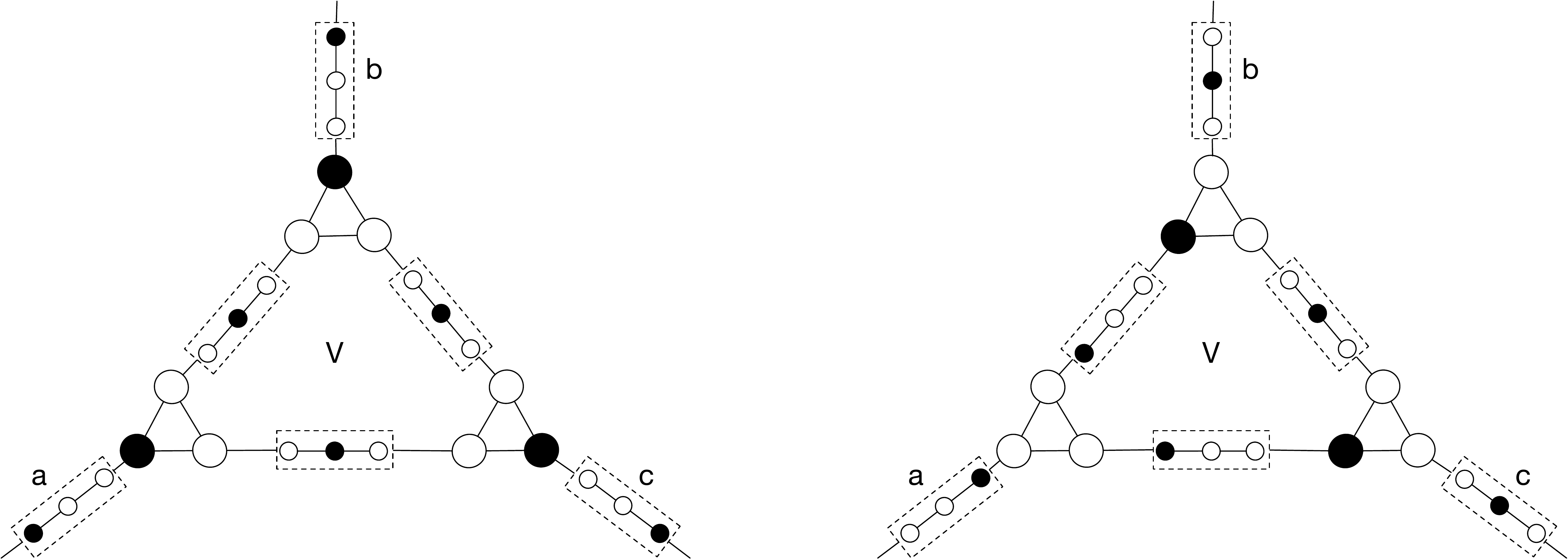}
\end{center}
\caption{On the left : the vertex $v$ is in the dominating set; on the right: the vertex $v$ is not in the dominating set.}
\label{kdoubgadj}
\end{figure}

The graph we obtain is $(claw,(3p-1)-double-triangle)$-free and its maximum degree is three. As shown by Figure \ref{kdoubgadj} the remaining arguments are the same as in the proof of Lemma \ref{butterfly}.

As $p$ is a positive integer that can be chosen arbitrary large we obtain the result.
\end{proof}

\section{$(claw,H)-free$ graphs}\label{clawH}
For any fixed graph $H$ with no more than six vertices we determine the complexity of the MDS problem for the class of $(claw,H)$-free graphs. Trivially the graphs $G$ we consider are assumed to be connected. When $H$ has no more than two vertices the MDS problem is trivially polynomial, so we prove the results for all the graphs $H$ with three vertices up to six vertices.

\begin{prop}\label{3vert}
For any fixed graph $H$ with three vertices computing a minimum dominating set for $G$ a connected $(claw,H)$-free graph is polynomial.
\end{prop}
\begin{proof}
Since a connected $(claw, C_3)$-free graph is either a path or a cycle, the MDS problem is polynomial when $G$ is $(claw, H)$-free with $H=C_3$. When $H\ne C_3$ we have $H\subseteq_i P_8$. From \cite{clawpk} we know that computing a mds is polynomial for $(claw, P_8)$-free graphs. Then the result holds from Property \ref{HinH'POLY}.
\end{proof}

The next three sections deal with the case where $H$ contains four, five, six vertices.
\subsection{$(claw,H)-free$ graphs when $H$ has four vertices}
When $H$ has exactly four vertices the state of the art is the following.
In \cite{Yannakakis}, the minimum edge dominating set problem is shown to be $NP$-complete for bipartite subcubic graphs. Moreover, one can check that the graphs $G$ built in the transformation (from a variant of the $3$-SAT problem) are $C_4$-free. Also, line graphs of bipartite graphs are perfect, so they have no odd holes. Thus, $L(G)$ the line graph of such a $G$ is $(claw, diamond, K_4,C_4)$-free and perfect (i.e., $(odd\ hole, odd\ antihole)$-free). It follows that the minimum dominating set problem is $NP$-complete for $(claw, diamond, K_4,C _4)$-free perfect graphs. From \cite{clawpk} computing a mds is polynomial for $(claw, P_k)$-free graphs when $k\le 8$. It is also polynomial when $H\in\{P_4, P_3+K_1,2K_2,K_2+2K_1,4K_1\}$ since $H\subseteq_i P_8$. The class of $(claw, \overline {claw})$-free graphs has bounded clique-width \cite{clawcoclaw}, so from \cite{cliquewidth} a $\gamma$-set can be computed in linear time (remark that $\overline {claw}\equiv K_3+K_1)$.\\

Taking these results together we obtain the following dichotomy.
\begin{prop}\label{4vert}
Let $H$ be a fixed graph with four vertices. Computing a minimum dominating set for $G$ a connected $(claw,H)$-free graph is $NP$-complete when $H\in\{diamond, K_4,C _4\}$, otherwise it is polynomial.
\end{prop}
\begin{rmk}
In \cite{indcomp} Lin et al. remark that the minimum dominating set problem is $NP$-complete for $(claw, diamond, K_4,C _4)$-free graphs, here we restrict this result for the class of perfect graphs.
\end{rmk}

\subsection{$(claw,H)-free$ graphs when $H$ has five vertices}
In this section, we focus on $H$ with exactly five vertices. From \cite{Yannakakis}, we know that the MDS problem is $NP$-complete when $G$ is $claw$-free and perfect, hence for $(claw,C_5)$-free graphs. Also, from Lemmas \ref{HinH'NPC} and \ref{butterfly}, we know that the MDS problem is $NP$-complete for $(claw, H)$-free graphs when $H$ contains a $C_4$, a $K_4$, a $diamond$, or a $butterfly$.

Using Property \ref{HinH'NPC} the MDS problem is NP-complete when $H\in \{K_5,K_5-e,\overline{P_3 +2K_1},\\ W_4,\overline{claw + K_1},\overline{P_2+ P_3},gem,\overline{K_3 + 2K_1},K_4 + K_1,C_4 + K_1,dart,house,diamond + K_1\}$.
From Property \ref{HinH'POLY}, it is polynomial for $(claw, H)$-free graphs when $H\subseteq_i net$ or $H \subseteq_i P_7$, and from Property \ref{kK1} it is polynomial for $(claw, k.K_1)$-free graphs.
Hence, the MDS problem is polynomial for $H\in \{bull,P_5,P_3+ 2K_1,2K_2+ K_1,K_2+ 3K_1,K_2 + P_3,P_4+ K_1,5K_1\}$.

It remains to treat the cases where $H\in \{paw+ K_1,\{2,0,0\}-triangle,K_3+ 2K_1,K_3+ K_2\}$.

\begin{prop}\label{K32K1}
Let $G=(V,E)$ be a connected $(claw,H)$-free graph. If $H=K_3 + 2K_1$ then computing a minimum dominating set for $G$ is polynomial.
\end{prop}
\begin{proof}
 When $G$ is a $(claw,K_3)$-free graph computing a mds is polynomial. Now we suppose that $K_3\subseteq_i G$. Let $G[\{v_1,v_2,v_3\}]=K_3$. From Property \ref{V=N[T]} we can assume that $N[K_3] \neq V$. Let $W=V-N[T]\ne\emptyset$. If $W$ is not complete then there are $u,v\in W$ such that $uv\not \in E$, but $K_3\cup \{u,v\}=H\subseteq_i G$. Thus $W$ is complete and $\gamma(G)\le 4$ (taking $\{v_1,v_2,v_3\}$ and $w\in W$). Hence a mds can be computed in time $O(n^4)$.\end{proof}

\begin{prop}\label{K3K2copan} 
Let $G=(V,E)$ be a connected $(claw,H)$-free graph. If $H\in\{K_3 + K_2,\{2,0,0\}-triangle,paw+K_1\}$ then computing a minimum dominating set for $G$ is polynomial.
\end{prop}
\begin{proof}
When $G$ is a $(claw,P_8)$-free graph computing a mds is polynomial (see \cite{clawpk}). Now we suppose that $P_8\subseteq_i G$. Let $v_1-v_2-v_3-v_4-v_5-v_6-v_6-v_7-v_8=P$. From Property \ref{V=N[T]} we can assume that $N[P] \neq V$. Let $W=V \setminus N[P]$. Since $G$ is connected and $claw$-free, there are $w \in W,v\in N(P)$ such that $v$ is a neighbor of $w$ and $v$ has exactly two neighbors $v_i, v_{i+1}$ in $P$. By symmetry we can assume that $i\leq 4$. Clearly $G[\{v,v_i,v_{i+1},v_{i+3}\}]=paw+K_1$, $G[\{v,v_i,v_{i+1}, v_{i+2},v_{i+3}\}]=\{2,0,0\}-triangle$ and $G[\{v,v_i,v_{i+1}, v_{i+3},v_{i+4}\}]=K_3+ K_2$.
\end{proof}

Taking the results together we obtain the following dichotomy.

\begin{theorem}\label{H5complexity}
 Let $H$ be a fixed graph with five vertices. Computing a minimum dominating set for $G$ a connected $(claw,H)$-free graph is $NP$-complete when $H \in \{C_5, K_5, K_5-e, \overline{P_3 + 2K_1}, W_4, \overline{claw + K_1}, \overline{P_2+ P_3},$ $gem, \overline{K_3 + 2K_1}, K_4 + K_1, C_4 + K_1, dart, house, diamond + K_1, butterfly\}$ and polynomial otherwise.
\end{theorem}

\subsection{$(claw,H)-free$ graphs when $H$ has six vertices}
We consider $H$ with exactly six vertices. From Theorems \ref{c4ck}, \ref{kdoubletriangle}, the MDS problem is $NP$-complete when $G$ is $(claw, C_4, \cdots, C_k)$-free (for any fixed $k\ge 4$) and when $G$ is $(claw, k-double-triangle)$-free (for any fixed $k\ge 0$). Hence, the MDS problem is $NP$-complete for $(claw, C_6)$-free graphs and for $(claw, double-triangle)$-free graphs. Furthermore, the MDS problem is $NP$-complete for $(claw, H)$-free graphs when $H$ contains at least one of the following graphs: $diamond, butterfly, claw, K_4$.

Also from lemma \ref{HinH'POLY}, the MDS problem is $polynomial$ for $(claw, H)$-free graphs when $H\subseteq_i H'$ and
the MDS problem is $polynomial$ for $(claw, H')$-free graphs. Hence, if $H$ is a net or is an induced subgraph of $P_8$, then the MDS problem is $polynomial$ for $(claw, H)$-free graphs.

So, we focus on $(claw, H)$-free graphs where the complexity of the MDS problem cannot be deduced from the previous arguments.

\begin{prop}\label{H6Poly}
Let $G=(V, E)$ a connected $(claw, H)$-free graph. Computing a minimum dominating set for $G$ is polynomial if $H\in\{K_3+P_3, \{3,0,0\}-triangle, \{2,0,0\}-triangle+K_1, 2K_3, P_3+3K_1, 2K_2+2K_1, paw+2K_1,bull+K_1, K_3+K_2+K_1, paw+K_2, \{2,1,0\}-triangle, K_2+4K_1, K_3+3K_1\}$.
\end{prop}
\begin{proof}
Let $P=v_1-\cdots-v_k$ a maximum induced path of $G$. Since computing a minimum dominating set is polynomial for $(claw,P_8)$-free graphs \cite{clawpk}, we can assume that $k\ge8$. Since $G$ is $claw$-free, any vertex $s\in N(P)$ is such that $\vert N_P(s) \vert \leq 4$. Also, from Property \ref{V=N[T]}, we can assume that $N[P] \neq V$, so $W=V-N[P]\ne\emptyset$. Let $w \in W$ with a neighbor $v \in V$ such that $N_P(v) \neq \emptyset$. Since $G$ is claw-free $v$ has exactly two neighbors that are consecutive in $P$, that is $N_P(v)=\{v_i, v_{i+1}\}$, $1 \leq i \leq k-1$. Since $P$ is a maximal path, these two neighbors cannot be $v_1,v_2$ or $v_{k-1},v_k$. Hence $N_P(v)=\{v_i, v_{i+1}\}$, $2 \leq i \leq k-2$. \\

From Property \ref{leaf} we can assume that both $v_1$ and $v_k$ are not leaves. Let $u_1, u_k \in V \setminus P$ be two neighbors of $v_1$ and $v_k$ respectively. We show how the vertices $u_1$ and $u_k$ are connected to $P$. First, we deal with $u_1=u_k$. From above $u_1$ has no neighbor in $W$. If $N_P(u_1)\supset\{v_1,v_2,v_3,v_k\}$ or $N_P(u_1)\supseteq\{v_1,v_{k-2},v_{k-1},v_k\}$ then $G$ has a claw. Hence, when $\vert N_P(u_1)\vert =4$, $N_P(u_1)=\{v_1,v_2,v_{k-1},v_k\}$ and when $\vert N_P(u_1)\vert =3$, $N_P(u_1)=\{v_1,v_2,v_k\}$ or $N_P(u_1)=\{v_1,v_{k-1},v_k\}$. Last, if $\vert N_p(u_1)\vert=2$, i.e. $N_p(u_1)=\{v_1,v_k\}$ then $w-v-v_i-\cdots v_1-u_1-v_k-\cdots-v_{i+1}=P_{k+2}$, a contradiction.
Hence there is no vertex $u\in V$ such that $N_P(u_1)=\{v_1, v_k\}$. Now when $u_1\ne u_k$ we have $N_P(u_1)=\{v_1,v_2\}$ or $N_P(u_1)=\{v_1,v_2,v_3\}$ or $N_P(u_1)=\{v_1,v_2,v_3,v_4\}$, and $N_P(u_k)=\{v_{k-1}, v_k\}$ or $N_P(u_k)=\{v_{k-2},v_{k-1},v_k\}$ or $N_P(u_k)=\{v_{k-3},v_{k-2},v_{k-1},v_{k}\}$. Hence, up to symmetry, in any case $G[\{v_1,v_2,u_1\}]$ is a triangle, and, $u_1$ and $u_k$ have no neighbor in $W$ \\

In the following we take $k=8$, the arguments being the same for $k>8$. Thus $N_P(v)=\{v_i, v_{i+1}\}$, $2 \leq i \leq 6$.

We demonstrate that $H\subseteq_i G$ when $P$ is of length $8$ thus for $k=8$. In the following, $i$ is the indice of the first neighbor of the vertex $v$ in $P$.

\begin{itemize}
    \item $H=K_3+P_3$: $G[\{u_1,v_1,v_2,v_5,v_6,v_7\}]=H$;
    \item $H=\{3,0,0\}-triangle$: when $2\leq i\leq 4$ then $G[\{v, v_{i}, v_{i+1}, v_{i+2}, v_{i+3}, v_{i+4}\}]=H$ (the case $5 \leq i \leq 6$ is symmetric);
    \item $H= \{2,0,0\}-triangle+K_1$: when $i\in \{2,3\}$ then $G[\{v, v_i, v_{i+1}, v_{i+2}, v_{i+3}, v_8\}]=H$ ($i \in \{5, 6\}$ is symmetric). When $i=4$ then $G[\{v, v_i, v_{i+1}, v_{i+2}, v_{i+3}, v_1\}]=H$;
    \item $2K_3$: when $2\leq i\leq 4$ then $u_8v\not\in E$ else $G$ contains a claw. Thus $G[\{v,v_i,v_{i+1}, v_{k-1}, v_k, u_k\}]=H$ (the case $i \in \{5, 6\}$ is symmetric);
    \item $H=P_3+3K_1$: $G[\{v_1, v_2, v_3, v_5, v_7, w\}]=H$;
    \item $H=2K_2+2K_1$: $G[\{v_1,v_2, v_4, v_5, v_7, w\}]=H$;
    \item $H=paw+2K_1$: when $i\in \{2,3\}$ then $G[\{v, v_i, v_{i+1}, w, v_6, v_8\}]=H$ (the case $i \in \{5, 6\}$ is symmetric). When $i=4$ then $G[\{v, v_i, v_{i+1},v_{i-1}, v_1, v_8\}]=H$;
    \item $H=bull+K_1$: when $2\leq i\leq 4$ then $G[\{v, v_i, v_{i+1}, v_{i-1},w,v_8\}]=H$ (the case $i\in \{5, 6\})$ is symmetric);
    \item $H=K_3+K_2+K_1$: $G[\{u_1, v_1, v_2, v_5, v_6, w\}]=H$;
    \item $H=paw+K_2$: when $2\leq i\leq 4$ then $G[\{v,v_i,v_{i+1},w,v_8\}]=H$ (the case $i\in \{5, 6\}$ is symmetric);
    \item $H=\{2,1,0\}-triangle$: when $2\leq i\leq 4$ then $G[\{v,v_{i-1},v_i,v_{i+1},v_{i+2},v_{i+3},w\}]=H$ (the case $i\in \{5, 6\}$ is symmetric);
    \item $H=K_2+4K_1$: $G[\{v_1, v_2, v_4, v_6, v_8, w\}]=H$;
    \item $H=K_3+3K_1$: $G[\{u_1, v_1, v_2, v_4, v_6, w\}]=H$.
\end{itemize}
\end{proof}
\subsection{The dichotomy  for the $H$'s with no more than six vertices}
Taking together all the results proved in this section we can conclude with the following result.
\begin{theorem}\label{H6catmost}
Let $G$ be a connected $(claw,H)$-free graph such that $H$ has at most six vertices. When $H'\subseteq_i H$ with $H'\in \{C_4,C_5,C_6,K_4,diamond,butterfly, double-triangle\}$, computing a minimum dominating set for $G$ is $NP$-complete, otherwise
 it is polynomial.
\end{theorem}

\section{More results for $(claw,H)-free$ graphs}\label{twotriangles}
In the previous section we give a complexity dichotomy when $H$ has no more than six vertices. Here we give a partial result for the $H$'s with more than six vertices.
\begin{theorem}
Let $k>0$ be a fixed integer and $H$ be a fixed graph with $k$ vertices such that $H$ has a connected component that contains two distinct triangles. The Minimum Dominating Set problem is $NP$-complete for $(claw, H)$-free graphs.
\end{theorem}
\begin{proof}
Let $T_1\subseteq_i H, T_2\subseteq_i H$ be two distinct triangles in a same component of $H$. When $T_1$ and $T_2$ are not vertex disjoint then $C_4\subseteq_i H$ or $diamond\subseteq_i H$ or $butterfly\subseteq_i H$. Now let $T_1$ and $T_2$ be vertex disjoint. When there is an edge from $T_1$ to $T_2$ then $double-triangle \subseteq_i H$ or $C_4\subseteq_i H$ or $K_4\subseteq_i H$ or $diamond\subseteq_i H$. If there is no edge between $T_1$ and $T_2$ then there exists $k'<k$ such that $k-double-triangle \subseteq_i H$. From our previous results we know that when $H\in \{C_4,K_4,diamond,butterfly, k'-double-triangle\}$ the MDS problem is $NP$-complete for $(claw, H)$-free graphs.
Thus the result follows from Property \ref{HinH'NPC}.
\end{proof}

\section{Conclusion}
We gave some complexity results for the Minimum Dominating Set problem for the class of $claw$-free graphs when another fixed graph $H$ is forbidden as an induced subgraph.  Especially we gave a complexity dichotomy for the class of $(claw,H)$-free graphs when $H$ has less than seven vertices. When $H$ has at least  seven vertices, we gave a partial result for the cases where $H$ contains a connected component with two triangles. For the other $H$'s we left the complexity problem open.

To reach our goal, as an intermediary step,  we proved the $NP$-completeness of the Minimum Dominating Set problem for the class of cubic graphs, which was unknown.


\begin{thebibliography}{99}
\bibitem{IndDom}
B. Allan, R. Laskar (1978), {\it On domination and independent domination numbers of a graph}, Discrete Mathematics 23, 73-76.

\bibitem{IndDomPoly}
B. Boliak, V. Lozin (2003) {\it Independent domination in finitely defined classes of graphs}, Theor. Comput. Scienc. 301, 271-284.

\bibitem{Bondy}
J.A. Bondy, U. S. R. Murty, \textit{Graphs Theory}, Springer, (2008).

\bibitem{clawpk}
V. Bouquet, C. Picouleau (2019) {\it The Minimum Dominating Set problem is polynomial for $(claw,P_8)$-free graphs}, Manuscript.

\bibitem{DomAlter}
D. Bauer, F. Harary, J. Nieminen, and C. L. Suffel (1983), Domination alteration sets in graphs. Discrete Mathematics, 47:153-161.

\bibitem{clawnet}A. Brandst\"adt, F. Dragan
 (2003) {\it On linear and circular structure of (claw, net)-free graph}, Discr. Appl. Math. 129 (2-3), 285-303.

\bibitem{clawcoclaw}
A. Brandst\" adt, S. Mahfud (2002) {\it Maximum weight stable set on graphs without claw and co-claw (and similar graph classes) can be solved in linear time.}, Inf. Proc. Let. 84, 251-259.

\bibitem{perfect}
M. Chudnovsky, N. Robertson, P. D. Seymour and R. Thomas, (2006) {\it The strong perfect
graph theorem}, Annals of Math. 164, 51-229.

\bibitem{cliquewidth}
B. Courcelle, J.A. Makowsky, U. Rotics (2000) {\it Linear time solvable optimization problems on graphs of bounded clique-width}, Theory Comput. Syst. 32 (2), 125-150.

\bibitem{Demange}
M. Demange, T. Ekim, C. Tanasescu (2014) {\it Hardness and approximation of minimum maximal matchings}, International Journal of Computer Mathematics 91 (8), 1635-1654.

\bibitem{GJ}
M. R. Garey and D. S. Johnson,  Computers and Intractability: A Guide to the Theory of NP-Completeness, Freeman, 1979.


\bibitem{DomBook}
Teresa W. Haynes, Stephen T. Hedetniemi, Peter J. Slater \textit{Fundamentals of Domination in Graphs}, Marcel Dekker Inc., (1998).

\bibitem{indcomp}
M. C. Lin, M. J. Mizrahi (2015), {\it On the complexity of the minimum domination problem restricted by forbidden induced subgraphs of small size}, Discr. Appl. Math. 197, 53-58.

\bibitem{Yannakakis}M. Yannakakis, F. Gavril (1980) {\it Edge dominating sets in graphs}, SIAM J. Appl. Math. 38 (3), 364-372.


\end{thebibliography}
\end{document}